\documentclass[12pt,twoside]{amsart} 

\usepackage{amssymb,latexsym,bbm,graphicx,epsfig,epic,eepic,oldgerm,psfrag,comment}
\usepackage{amsmath,amsfonts,amscd,mathrsfs,amsthm}
\usepackage{mathtools}
\usepackage{mathtools}
\usepackage{forest}

\usepackage{tikz,wasysym}  \usetikzlibrary{trees}
\usepackage{marvosym}                  
\usepackage{hyperref} 

\usepackage{xypic} 
\input xy
\xyoption{all}

\theoremstyle{plain}
\newtheorem{theorem}{Theorem}[section]
\newtheorem{lemma}[theorem]{Lemma}                           
\newtheorem{proposition}[theorem]{Proposition}
\newtheorem{corollary}[theorem]{Corollary}
\newtheorem*{remark*}{Remark}
\newtheorem*{remarks*}{Remarks}
\newtheorem{remark}[theorem]{Remark}

\newtheorem*{example*}{Example}
\newtheorem*{examples*}{Examples}

\newtheorem*{definition*}{Definition}

\newtheorem*{claim*}{Claim}

\numberwithin{figure}{section}
\numberwithin{equation}{section}

\newcommand{\proofend}{\hspace*{\fill} $\square$\\}

\def\1{\:\!}
\def\2{\;\!}

\def\Diffc0{\operatorname{Diff^c_0}}

\def\Sympc0{\operatorname{Symp^c_0}}

\def\CC{\mathbb{C}}

\def\RR{\mathbb{R}}

\def\ZZ{\mathbb{Z}}

\def\R{\operatorname{\mathbb{R}}}

\begin{document}

\title{\vspace*{0cm} On the agreement of symplectic capacities in high dimension}
\author{Dan Cristofaro-Gardiner and Richard Hind}

\date{\today}

\maketitle

\begin{abstract}

A theorem of Gutt-Hutchings-Ramos asserts that all normalized symplectic capacities give the same value for monotone four-dimensional toric domains.
We generalize this theorem to arbitrary dimension.  The new ingredient in our proof is the construction of symplectic embeddings of  
 ``$L$-shaped" domains in any dimension into corresponding infinite cylinders; this resolves a conjecture of Gutt-Pereira-Ramos in the affirmative.

\end{abstract}

\section{Introduction}

Symplectic capacities are measurements of symplectic size: a {\em symplectic capacity} is a rule that associates to each symplectic manifold $(M,\omega)$ of a fixed dimension a nonnegative, possibly infinite, real number $c(M,\omega)$.  The assignation $(M,\omega) \to c(M,\omega)$ is subject to various axioms; perhaps the most important is the ``Monotonicity Axiom", which states that 
\[ c(M_1,\omega_1) \le c(M_2,\omega_2),\]
when $(M_1,\omega_1)$ can be symplectically embedded into $(M_2,\omega_2)$; for the other axioms, we refer the reader to \cite{chls}.  Any scalar multiple of a symplectic capacity is itself a symplectic capacity.  However, a symplectic capacity $c$ is said to be {\em normalized}  if
\[ c(B^{2n}(1)) = c(Z^{2n}(1)) = 1,\]
where $B^{2n}$ denotes the ball and $Z^{2n}$ denotes the infinite cylinder, both defined below.

While there are many symplectic capacities (see e.g. \cite{chls}), it is interesting to try to understand to what degree normalized symplectic capacities are unique.  For example, 
a conjecture due to Viterbo states that all normalized capacities give the same value on convex domains in $\mathbb{R}^{2n}$.  This is known to imply the famous Mahler conjecture \cite{vm}.   Another instance of this kind of phenomenon, which is the genesis of the current note, is the following theorem of Gutt-Hutchings-Ramos.  To set the notation,    let $\pi: \CC^n \to \R^n_{\ge 0}$ be given by $\pi(z_1, \dots , z_n) = (\pi|z_1|^2, \dots , \pi|z_n|^2)$. Then given a subset $\Omega \subset \R^n_{\ge 0}$ we can define the associated {\em toric domain} by
$$X_{\Omega} := \pi^{-1}(\Omega) \subset \CC^n.$$
The vector space $\CC^n$ has a standard symplectic structure $\omega = \frac{1}{2i} \sum dz_k \wedge d\overline{z_k}$ and so the $X_{\Omega}$ inherit a symplectic structure.
A toric domain $X_{\Omega}$ is called {\em monotone} if the outward normals at every point $\mu \in \partial \Omega \cap \R^n_{> 0}$ have nonnegative entries.  Many monotone toric domains are not convex, and many convex domains are not toric.  


\begin{theorem}[\cite{ghr}]
Let $X_{\Omega}$ be a four-dimensional monotone toric domain.  Then any two normalized symplectic capacities $c_1$ and $c_2$ satisfy 
\[ c_1(X_{\Omega}) = c_2(X_{\Omega}).\]
\end{theorem}


Our main result extends this theorem to arbitrary dimension.

\begin{theorem}\label{agree}
All normalized symplectic capacities agree on monotone toric domains of any dimension.
\end{theorem}

A strictly monotone toric domain is one for which the outward normals at every point   $\mu \in \partial \Omega \cap \R^n_{\ge 0}$ have positive entries. It was shown in \cite{ghr}, Proposition 1.8, that strictly monotone toric domains are dynamically convex. The notion of dynamical convexity is invariant under symplectomorphisms, and all strictly convex domains are dynamically convex, see \cite{convex}, section 3. Abbondandolo, Bramham, Hryniewicz, and Salam\~{a}o \cite{abhs} have constructed 4 dimensional dynamically convex domains for which Viterbo's conjecture does not hold, see \cite{ghr} Remark 1.9, but we do not know if similar examples exist in higher dimension.
 
Let us now explain the key new ingredient in our proof, which is a theorem of potentially independent interest.
Given $r>0$ we define the ball and cylinder of capacity $r$ by
$$B(r) = X_{b(r)}, \qquad Z(r) = X_{z(r)}$$
respectively, where $b(r) = \{ x_k \ge 0 \, | \, \sum x_k < r\}$ and $z(r) = \{x_n < r\}$.
Also, given $a_1, \dots , a_n >0$, we define an {\em $L$-shaped domain} by
$$L(a_1, \dots , a_n) = X_{l(a_1, \dots , a_n)}$$
where $l(a_1, \dots , a_n) = \bigcup_{k=1}^n \{ x_k < a_k \}$.
A crucial step in the argument of Gutt-Hutchings-Ramos is to show that $L(a_1,a_2)$ can be symplectically embedded into $Z(a_1+ a_2)$.  This is proved by making use of \cite{cg}, a rather general embedding result proved using pseudoholomorphic curves and symplectic inflation.  As the techniques in \cite{cg} have no known analogue in higher dimensions, it is natural to wonder whether or not the corresponding embeddings of $L$-shaped domains in higher dimensions exist; our second result resolves this. 






\begin{theorem}\label{main} Let $r > \sum a_k$. Then there exists a symplectic embedding
$$L(a_1, \dots , a_n) \hookrightarrow Z(r).$$
\end{theorem}

A slightly simplified version of this was conjectured in all dimensions by Gutt, Periera and Ramos in \cite[Conj. 15]{gpr}. 

\begin{remark} Combining the proof of Theorem \ref{main} with Theorem 4.3 in \cite{pvn}, it is actually possible to produce embeddings even in the case when $r = \sum a_k$. However the strict inequality is enough to derive our consequences for symplectic capacities.
\end{remark}



Let us now recall why Theorem \ref{main} implies Theorem \ref{agree}.

\begin{proof}  Given $X \subset \CC^n$ we define its Gromov width by $$c_G(X) = \sup \{ r>0 \, | \, B(r) \hookrightarrow X \}$$ and its cylindrical capacity by $$c_Z(X) = \inf \{ r>0 \, | \, X \hookrightarrow Z(r) \}.$$ The definition of normalized symplectic capacities implies that  $$c_G(X) \le c(X) \le c_Z(X)$$ for all such capacities $c$ and all $X \subset \CC^n$.

Thus, it suffices to show that  $c_G(X) = c_Z(X)$ for all monotone toric domains $X_{\Omega}$.


Let $B(r)$ be the largest toric ball that embeds by inclusion.   Then there is a point $(a_1,\ldots,a_n) \in \partial \Omega$ with $a_1 + \ldots + a_n = r.$
It now suffices to show that $r$ is an upper bound on the cylindrical capacity. Because $\Omega$ is monotone, $X_{\Omega} \subset L(a_1,\ldots,a_n).$  By Theorem~\ref{main}, it follows that there is a symplectic embedding 
$X_{\Omega} \to Z(r)$
Thus, it follows that the cylindrical capacity is at most $r$, as desired.





\proofend
\end{proof}

{\bf Acknowledgements.} We would like to thank Jean Gutt, Michael Hutchings and Vinicius Ramos for helpful discussions.
DCG also thanks the National Science Foundation for their support under agreement DMS-2227372, and RH thanks the Simons Foundation for their support under grant no. 633715.  We also thank the Brin Mathematics Research Center at the University of Maryland for hosting a visit by RH, during which important conversations about this project occurred.


\section{Construction of a symplectic embedding}

Our symplectic embedding is quite explicit, exploiting an idea from \cite[Sec. 2.1]{busehind11}.
In section \ref{two1}  we describe a general construction for embedding domains in $\CC^n$ which are invariant under the diagonal $S^1$ action. Specifically we reduce the construction to finding a 1 parameter family of embeddings in complex $(n-1)$ dimensional projective space. In section \ref{two2} we apply the construction to the domains $L(a_1, \dots , a_n)$ and $Z(r)$. This reduces Theorem \ref{main} to finding a family of Hamiltonian diffeomorphisms of $\CC P^{n-1}$, which we proceed to do.

\subsection{Symplectomorphisms of projective space and embeddings in $\CC^n$}\label{two1}

Let $p_S: \partial B(S) \to \CC P^{n-1}(S)$ be the symplectic reduction, that is, the quotient by the characteristic orbits. Hence $\CC P^{n-1}(S)$ is complex projective space and inherits a symplectic form which integrates to $S$ over complex lines.

Next let $H : \CC P^{n-1}(S) \to \RR$ and $\tilde{H}: \CC^n \to \RR$ be a smooth extension of $H \circ p_S : \partial B(S) \to \RR$. Denote by $\phi$ and $\tilde{\phi}$ the corresponding Hamiltonian diffeomorphisms of $\CC P^{n-1}(S)$ and $\CC^n$ respectively.

\begin{lemma} \label{lift}
If $z \in \partial B(S)$ we have $\tilde{\phi}(z) \in \partial B(S)$ and $p_S (\tilde{\phi}(z)) = \phi(p_S(z))$.
\end{lemma}

\begin{proof} As $\tilde{H}$ is constant along the characteristic orbits, the Hamiltonian vector field $X_{\tilde{H}}$ is tangent to $\partial B(S)$ and hence its flow preserves $\partial B(S)$, giving the first part of the statement. Also, if $z \in \partial B(S)$ we have $dp_S \, X_{\tilde{H}}(z) = X_H(p_S(z))$, and integrating gives the second part.
\proofend
\end{proof}

The main result of this section will be a parameterized version of Lemma \ref{lift}. Let $$p: \CC^n \setminus \{ 0 \} \to (0,\infty) \times \CC P^{n-1},$$ $$z \in \partial B(S) \mapsto (S, p_S(z)).$$
A family of functions $H_S : \CC P^{n-1} \to \R$ define a function $H: (0,\infty) \times \CC P^{n-1} \to \R$ by $(S, z) \mapsto H_S(z)$ which we always assume to be smooth. We will also assume there is a constant $c$ so that $H_S(z) = c$ whenever $S$ is small. Then $\tilde{H} = H \circ p$ extends smoothly to a function on $\CC^n$ with $\tilde{H}(0)=c$. It is invariant under the diagonal $S^1$ action generated by multiplication by the unit circle, that is, $\tilde{H}(e^{it}z) = \tilde{H}(z)$.


As in the Lemma \ref{lift}, let $\phi_S$ be the Hamiltonian diffeomorphism of $\CC P^{n-1}(S)$ generated by $H_S$ and $\tilde{\phi}$ be the  Hamiltonian diffeomorphism of  $\CC^n$ generated by $\tilde{H}$.

Let $U, V \subset \CC^n$ be open sets with $V$ invariant under the $S^1$ action and with $0 \in V$, and let $U_S, V_S \subset \CC P^{n-1}(S)$ be the images of $U \cap \partial B(S)$ and $V \cap \partial B(S)$ under the projection maps $p_S : \partial B(S) \to \CC P^{n-1}(S)$. Then Lemma \ref{lift} implies the following.

\begin{corollary}\label{embed} Suppose $\phi_S(U_S) \subset V_S$ for all $S$. Then $\tilde{\phi}(U) \subset V$.
\end{corollary}

\begin{proof} We have $\tilde{\phi}(0) = 0 \in V$, so it suffices to check $\tilde{\phi}(z) \in V$ for $z \in U \setminus \{ 0 \}$. Suppose then that $z \in \partial B(S)$ with $S>0$. Then Lemma \ref{lift} says that
$$\tilde{\phi}(z) \in p^{-1}(S, \phi_S(p_S(z))) \subset p^{-1}(\{S \} \times  \phi_S(U_S)) $$ $$ \subset p^{-1} (\{S\} \times V_S) = \partial B(S) \cap V$$
as required, where the final equality follows from the $S^1$ invariance of $V$.

\end{proof}
\proofend

\subsection{An embedding of L-shaped domains}\label{two2}

In this subsection Corollary \ref{embed} is applied to prove Theorem \ref{main}, which is restated here for convenience.

\begin{proposition} Suppose $r > a_1 + \dots a_n$. Then there exists a symplectic embedding $L(a_1, \dots, a_n) \hookrightarrow Z(r)$.
\end{proposition}

\begin{proof}
{\bf Step 1.} In this step we translate the general sufficient conditions from section \ref{two1} into specific requirements for our embedding.

Following the notation from Corollary \ref{embed} we set $U = L(a_1, \dots, a_n)$ and $V = Z(r)$. These domains are invariant under our $S^1$ action, and in fact are toric, so for $S>0$ they project to toric domains $U_S, V_S \subset \CC P^{n-1}(S)$. We note that the $T^n$ action on $\CC^n$ descends to the $\CC P^{n-1}(S)$, but now has a $1$ dimensional kernel generated by the diagonal action $z \mapsto e^{it}z$. 

We can use polar coordinates $R_i = \pi|z_i|^2$, $\theta_i \in \RR / \ZZ$ on $\CC^n$, so $\partial B(S) = \{ \sum R_i = S \}$. We then identify the symplectic reduction of $\partial B(S)$ with the toric manifold associated to the projection of this set to the $(R_1, \dots , R_{n-1})$ plane, a closed triangle $\Delta_S$ with vertices $(0, \dots, 0,S,0, \dots, 0)$. Under this identification, the $T^n$ fiber over $(R_1, \dots , R_n)$ in $\CC^n$ projects to the toric fiber over $(R_1, \dots , R_{n-1})$ via the map $$(\theta_1, \dots , \theta_n) \mapsto (\theta_1 - \theta_n, \dots , \theta_{n-1} - \theta_n).$$ 

Let $\mu : \CC P^{n-1}(S) \to \Delta_S$ be the moment projection.
We have
$$\mu(U_S) = \bigcup_{i=1}^{n-1} \{R_i < a_i \} \bigcup \{ \sum_{i=1}^{n-1} R_i > S - a_n\}$$
and 
$$\mu(V_S) = \{ \sum_{i=1}^{n-1} R_i > S-r\}.$$

By Corollary \ref{embed}, to find a symplectic embedding $U \hookrightarrow V$ it suffices to find a smooth family of Hamiltonian functions $H_S$ on $\CC P^{n-1}(S)$ generating diffeomorphisms $\phi_S$ with  $\phi_S(U_S) \subset V_S$ for all $S$. The condition is vacuous when $S \le r$, when $V_S = \CC P^{n-1}(S)$, and so it suffices to construct $H_S$ for $S> r - \epsilon$, say, and then apply a bump function so that $H_S(z)=0$ whenever $S$ is small. 

Equivalently, by looking at the complements, it suffices to construct a family of Hamiltonian diffeomorphisms mapping $$\overline{B(S-r)} = \CC P^{n-1}(S) \setminus V_S \hookrightarrow \CC P^{n-1}(S) \setminus U_S = \{R_i \ge a_i \, \forall i , \, \sum R_i \le S - a_n\}.$$

{\bf Step 2.} In this step we show that Hamiltonian diffeomorphisms as in Step 1 do indeed exist, and we can even arrange the support to lie in the affine part $\{\sum_{i=1}^{n-1} R_i < S\} \subset \CC P^{n-1}$. We describe the diffeomorphisms for a fixed $S$, however it will be clear the generating functions can be chosen to depend smoothly on $S$ (after identifying our models of $\CC P^{n-1}(S)$ with the fixed underlying manifold).

We define $D(T) = \{ \pi|z|^2 \le T\}$ to be the closed disk in the $z$ plane, and $A(T_1, T_2) = \{ T_1 < \pi|z|^2 < T_2\}$ to be the open annulus.

Let $\phi_i$, $1 \le i \le n-1$, be Hamiltonian diffeomorphisms of the plane 
such that  $$\phi_i(D(S-r)) \subset A(a_i, S- \sum_{j \neq i} a_j).$$ 
We note that such Hamiltonians exist since $r > \sum a_i$ (implying that the disk has strictly smaller area than the annulus). Moreover we can arrange that a Hamiltonian flow $\phi_i^t$ with $\phi_i^1 = \phi_i$ satisfies $\pi |\phi^t_i(z)|^2 < \pi|z|^2 + ta_i + \epsilon$ for all $0 \le t \le 1$ and $\epsilon$ arbitrarily small.

Now consider the Hamiltonian flow of $\CC^{n-1}$ given by  $$\Phi^t: (z_1, \dots z_{n-1}) \mapsto (\phi^t_1(z_1), \dots \phi^t_{n-1}(z_{n-1})).$$
Our proof will follow from properties of $\Phi^t$.

{\bf Claim.} 
\begin{enumerate}
\item $\Phi^1(\overline{B(S-r)}) \subset \{\pi|z_i|^2 > a_i \, \forall i , \, \sum \pi|z_i|^2 < S - a_n\}$.
\item $\Phi^t(\overline{B(S-r)}) \subset \{\sum \pi|z_i|^2 < S - a_n \}$ for all $0 \le t \le 1$.
\end{enumerate}

The second part of the claim implies that $\Phi^t$ has a generating Hamiltonian function which can be cut off to have support inside $B(S)$. Hence there exists a Hamiltonian diffeomorphism of $\CC P^2(S)$ which acts on $\overline{B(S-r)}$ in the same way as $\Phi^1$. In particular the first part of the claim then implies there exists a Hamiltonian diffeomorphism of $\CC P^2(S)$ mapping $\overline{B(s-r)} = \CC P^2 \setminus V_S$ into $\CC P^2(S) \setminus U_S$ as required.

To justify the claim we suppose $(z_1, \dots z_{n-1}) \in \overline{B(S-r)}$.
Then we have $$\sum_{i=1}^{n-1} \pi |\phi^t_i (z_i)|^2 < \sum (\pi |z_i|^2 + ta_i + \epsilon) \le S - r + (n-1)\epsilon + t\sum a_i < S - a_n$$ when $\epsilon$ is sufficiently small. Here the first inequality follows from our assumptions on the $\phi^t_i$, the second holds since $\sum \pi|z_i|^2 \le S-r$ and the final innequality uses $\sum a_i < r$. This establishes statement (2). For the remainder of (1) we just recall that $\phi^1_i(z) = \phi_i(z) \subset \{ \pi |z|^2 > a_i\}$.



\proofend

\end{proof}











\end{document}